\documentclass[preprint,nonatbib]{elsarticle}
\usepackage{fullpage} 

\makeatletter
\let\c@author\relax
\makeatother

\usepackage[table]{xcolor}
\usepackage{amsmath}
\usepackage{amsthm}
\usepackage{amssymb} 
\usepackage{hyperref}
\usepackage{graphicx}
\usepackage{float}
\usepackage{tikz-cd}
\usepackage{caption}
\usepackage{subcaption}
\usepackage{array}
\usepackage{extarrows}
\usepackage{multirow}
\usepackage{tikz}
\usepackage{makecell}
\usepackage{hhline}
\usepackage{ mathdots }
\usepackage{cleveref}
\usetikzlibrary{cd}
\usetikzlibrary{graphs}
\usetikzlibrary {graphs.standard}
\usetikzlibrary{shapes.geometric}

\DeclareMathOperator{\LS}{2-RP}

\allowdisplaybreaks[3]

\newtheorem{theorem}{Theorem}[section]
\newtheorem{lemma}[theorem]{Lemma}

\newtheorem{conjecture}[theorem]{Conjecture}

\theoremstyle{definition}
\newtheorem{definition}[theorem]{Definition}
\newtheorem{example}[theorem]{Ex.}
\newtheorem{construction}[theorem]{Construction}

\numberwithin{equation}{section}

\usepackage{biblatex} 
\hypersetup{pdfauthor=author}

\makeatletter
\def\ps@pprintTitle{%
  \let\@oddhead\@empty
  \let\@evenhead\@empty
  \def\@oddfoot{\reset@font\hfil\thepage\hfil}
  \let\@evenfoot\@oddfoot
}
\makeatother

\begin{document}

\title{Realizations with five subsquares}
\author{Tara Kemp}
\ead{t.kemp@uq.net.au}

\affiliation{organization={School of Mathematics and Physics, ARC Centre of Excellence, Plant Success in Nature and Agriculture},
    addressline={The University of Queensland},
    city={Brisbane},
    postcode={4072},
    country={Australia}}

\begin{abstract}
    Given an integer partition $(h_1,h_2,\dots,h_k)$ of $n$, is it possible to find an order $n$ latin square with $k$ disjoint subsquares of orders $h_1,\dots,h_k$? This question was posed by L.Fuchs and is only partially solved. Existence has been determined in general when $k\leq 4$, and in this paper we will complete the case when $k=5$. We also prove some less general results for partitions with $k=5$.
\end{abstract}


\maketitle

\section{Introduction}

For an integer partition $P = (h_1,\dots,h_k)$, a realization of $P$ is a latin square with $k$ disjoint subsquares whose orders are $h_1,\dots,h_k$. The problem of existence of a realization was originally posed by L.Fuchs \cite{keedwell2015latin} as a question about quasigroups with disjoint subquasigroups, however, this is equivalent to a realization or a partitioned incomplete latin square (PILS). The problem is only partially solved, with existence completely determined for partitions with at most four parts or with at most two distinct parts.

\begin{definition}
    A \emph{latin square} of order $n$ is an $n\times n$ array, $L$, on a set of $n$ symbols such that each symbol occurs exactly once in each row and column. An entry of $L$ is defined by its row and column position as $L(i,j)$.

    An $m\times m$ sub-array of $L$ that is itself a latin square of order $m$ is a \emph{subsquare} of $L$. Subsquares are disjoint if they share no row, column or symbol of $L$.
\end{definition}

\begin{example}
\label{order8square}
A latin square of order 8 is given in \Cref{square example}, with subsquares of orders 1, 2 and 3 highlighted.
\begin{figure}[h]
    \centering
    $\arraycolsep=4pt\begin{array}{|c|c|c|c|c|c|c|c|} \hline
    \cellcolor{lightgray}1 & \cellcolor{lightgray}2 & \cellcolor{lightgray}3 & 6 & 7 & 8 & 5 & 4\\ \hline
    \cellcolor{lightgray}2 & \cellcolor{lightgray}3 & \cellcolor{lightgray}1 & 7 & 8 & 4 & 6 & 5\\ \hline
    \cellcolor{lightgray}3 & \cellcolor{lightgray}1 & \cellcolor{lightgray}2 & 8 & 6 & 5 & 4 & 7\\ \hline
    6 & 7 & 8 & \cellcolor{lightgray}4 & \cellcolor{lightgray}5 & 2 & 3 & 1\\ \hline
    7 & 8 & 6 & \cellcolor{lightgray}5 & \cellcolor{lightgray}4 & 3 & 1 & 2\\ \hline
    4 & 5 & 7 & 1 & 2 & \cellcolor{lightgray}6 & 8 & 3\\ \hline
    8 & 4 & 5 & 2 & 3 & 1 & \cellcolor{lightgray}7 & 6\\ \hline
    5 & 6 & 4 & 3 & 1 & 7 & 2 & \cellcolor{lightgray}8\\ \hline
    \end{array}$
    \caption{A latin square of order 8 with subsquares highlighted}
    \label{square example}
\end{figure}
\end{example}

We will assume that all latin squares of order $n$ use the symbols of $[n]$, where $[n]$ is the set of integers $\{1,2,\dots,n\}$. Given an integer $a$, let $a+[n]$ be the set $\{a+1,a+2,\dots,a+n\}$.

\begin{definition}
    An \emph{integer partition} of $n$ is a sequence of non-increasing integers $P = (h_1,h_2,h_3,\dots)$ where $\sum_{i=1}^\infty h_i = n$.
\end{definition}

We will represent an integer partition $P = (h_1,h_2,\dots,h_k)$ by $(h_1h_2\dots h_k)$ or $(h_{i_1}^{\alpha_1}h_{i_2}^{\alpha_2}\dots h_{i_m}^{\alpha_m})$, where there are $\alpha_j$ copies of $h_{i_j}$ in the partition.

\begin{definition}
     For an integer partition $P = (h_1,h_2,\dots,h_k)$ of $n$ with $h_1\geq h_2\geq \dots\geq h_k > 0$, a \emph{2-realization} of $P$, denoted $\LS(h_1h_2\dots h_k)$, is a latin square of order $n$ with pairwise disjoint subsquares of order $h_i$ for each $i\in [k]$.
     We say that a 2-realization is in \emph{normal form} if the subsquares are along the main diagonal and the $i^{th}$ subsquare, $H_i$, is of order $h_i$.
\end{definition}

A 2-realization of the partition $(3^12^11^3)$ is shown in \Cref{order8square}.

The existence of a 2-realization for a given partition is only a partially solved problem. We will now provide some of the known results. For partitions with a small number of parts, existence has been determined in \cite{heinrich2006latin}.

\begin{theorem}
\label{small n squares}
    Take a partition $(h_1,h_2,\dots,h_k)$ of $n$ with $h_1\geq h_2\geq \dots\geq h_k > 0$. Then a $\LS(h_1h_2\dots h_k)$
    \begin{itemize}
        \item always exists when $k=1$;
        \item never exists when $k=2$;
        \item exists when $k=3$ if and only if $h_1=h_2=h_3$;
        \item exists when $k=4$ if and only if $h_1=h_2=h_3$, or $h_2=h_3=h_4$ with $h_1\leq 2h_4$.
    \end{itemize}
\end{theorem}

There are some known results for specific cases of partitions when $k=5$, which can be found in \cite{colbourn2018latin} and \cite{kuhl2019existence}. In this paper we will complete the general $k=5$ case.

Existence has also been determined when the subsquares are of at most two distinct orders. \Cref{a^n square} is from D{\'e}nes and P{\'a}sztor \cite{denes1963some}, and \Cref{squaresatmost2} is a combination of results from Heinrich \cite{heinrich1982disjoint} and Kuhl et al. \cite{kuhl2018latin}, which completes the cases not covered above.
\begin{theorem}
\label{a^n square}
    For $k\geq 1$ and $a\geq 1$, a $\LS(a^k)$ exists if and only if $k\neq 2$.
\end{theorem}

\begin{theorem}
\label{squaresatmost2}
    For $a>b>0$ and $k>4$, a $\LS(a^ub^{k-u})$ exists if and only if $u\geq 3$, or $0< u < 3$ and $a\leq (k-2)b$.
\end{theorem}

There are two necessary conditions for the existence of a 2-realization from Colbourn \cite{colbourn2018latin}, however they were shown by Colbourn to not be sufficient.
\begin{theorem}
\label{squarecondition1}
    If a $\LS(h_1h_2\dots h_k)$ exists, then $h_1\leq\sum_{i=3}^kh_i$.
\end{theorem}

\begin{theorem}
\label{squarecondition2}
    If a $\LS(h_1h_2\dots h_k)$ exists, then $$n^2 - \sum_{i=1}^kh_i^2\geq 3\Bigg( \sum_{i\in D}h_i\Bigg)\Bigg(\sum_{j\in\overline{D}}h_j\Bigg)$$ for any $D\subseteq\{1,2,\dots,k\}$ and $\overline{D} = \{1,2,\dots,k\}\setminus D$.
\end{theorem}

\section{Constructions}

In this section, we will introduce the methods that will be used to prove our main result.

\begin{definition}
    Given partitions $P,Q,R$ of $n$, where $P = (p_1,\dots,p_u)$, $Q = (q_1,\dots,q_v)$, $R = (r_1,\dots,r_t)$, let $O$ be a $u\times v$ array of multisets, with elements from $[t]$. For $i\in [u]$ and $j\in[v]$, let $O(i,j)$ be the multiset of symbols in cell $(i,j)$ and let $|O(i,j)|$ be the number of symbols in the cell, including repetition.

    Then $O$ is an \emph{outline rectangle} associated to $(P,Q,R)$ if
    \begin{enumerate}
        \item $|O(i,j)| = p_iq_j$, for all $(i,j)\in[u]\times[v]$;
        \item symbol $l\in[t]$ occurs $p_ir_l$ times in the row $(i,[v])$;
        \item symbol $l\in[t]$ occurs $q_jr_l$ times in the column $([u],j)$.
    \end{enumerate}
\end{definition}

\begin{example}
The array in \Cref{outline square} is an outline rectangle associated to $(P,Q,R)$, for partitions $P=(5^12^11^1)$, $Q=(4^12^2)$ and $R=(3^21^2)$.
    \begin{figure}[h]
        \centering
        $\arraycolsep=4pt\begin{array}{|cccc|cc|cc|} \hline
        1 & 1 & 1 & 1 & 1 & 1 & 1 & 1\\ 
        1 & 1 & 1 & 1 & 2 & 2 & 1 & 1\\ 
        1 & 2 & 2 & 2 & 2 & 2 & 2 & 2\\ 
        2 & 2 & 3 & 3 & 2 & 3 & 2 & 2\\ 
        3 & 4 & 4 & 4 & 4 & 4 & 2 & 3\\ \hline
        1 & 1 & 2 & 2 & 1 & 1 & 1 & 2\\ 
        2 & 2 & 3 & 4 & 1 & 2 & 3 & 4\\ \hline
        1 & 2 & 2 & 2 & 1 & 3 & 1 & 4\\ \hline
        \end{array}$
        \caption{An outline rectangle associated to $(5^12^11^1,4^12^2,3^21^2)$.}
        \label{outline square}
    \end{figure}
\end{example}

\begin{definition}
    Given partitions $P,Q,R$ of $n$, where $P = (p_1,\dots,p_u)$, $Q = (q_1,\dots,q_v)$ and $R = (r_1,\dots,r_w)$, and a latin square $L$ of order $n$, the \emph{reduction modulo $(P,Q,R)$} of $L$, denoted $O$, is the $u\times v$ array of multisets obtained by amalgamating rows $(p_1 + \dots + p_{i-1}) + [p_i]$ for all $i\in[u]$, columns $(q_1 + \dots + q_{j-1}) + [q_j]$ for all $j\in[v]$, and symbols $(r_1 + \dots + r_{k-1}) + [r_k]$ for all $k\in[w]$.

    When amalgamating symbols, for $k\in[w]$ we will map all symbols in $(r_1 + \dots + r_{k-1}) + [r_k]$ to symbol $k$.
\end{definition}

\begin{example}
The outline rectangle in \Cref{outline square} is a reduction modulo $(P,Q,R)$ of the latin square in \Cref{order8square}, where $P=(5^12^11^1)$, $Q=(4^12^2)$ and $R=(3^21^2)$.
\end{example}

If an outline rectangle $O$ is a reduction modulo $(P,Q,R)$ of a latin square $L$, then we say that $O$ \emph{lifts} to $L$.

Hilton \cite{hilton1980reconstruction} has proven the following theorem.

\begin{theorem}
\label{outline rectangle to square}
    Let $P,Q,R$ be partitions of $n$. For every outline rectangle $O$ associated to $(P,Q,R)$, there is a latin square $L$ of order $n$ such that $O$ lifts to $L$.
\end{theorem}

This theorem is an important result, since it means that constructing an outline rectangle is enough to prove the existence of a latin square. In order to find realizations, we give the following construction, which has been used before in other work.

\begin{construction}
    For a partition $P = (h_1,\dots,h_k)$, an outline rectangle modulo $(P,P,P)$ with cell $(i,i)$ filled with $h_i^2$ copies of symbol $i$ for all $i\in[k]$ lifts to a 2-realization of $P$.
\end{construction}

A \emph{rational outline rectangle} is similar to an outline rectangle, except that the number of copies of a symbol in a cell is a non-negative rational number.

\begin{definition}
    For the partitions $P,Q,R$ of $n$, where $P = (p_1,\dots,p_u)$, $Q = (q_1,\dots,q_v)$, $R = (r_1,\dots,r_t)$, let $O_k(i,j)$ be a non-negative rational number for all $i\in[u]$, $j\in[v]$ and $k\in[t]$.

    Then the function $O_k(i,j)$ forms a \emph{rational outline rectangle} associated to $(P,Q,R)$ if
    \begin{enumerate}
        \item $\sum_{k\in[t]} O_k(i,j) = p_iq_j$, for all $(i,j)\in[u]\times[v]$;
        \item $\sum_{j\in[v]} O_k(i,j) = p_ir_k$, for all $i\in[u]$ and $k\in[t]$;
        \item $\sum_{i\in[u]} O_k(i,j) = q_jr_k$, for all $j\in[v]$ and $k\in[t]$.
    \end{enumerate}
\end{definition}

\begin{example}
    \Cref{rational outline rectangle} gives a rational outline rectangle for the partitions $P=(5^12^11^1)$, $Q=(4^12^2)$ and $R=(3^21^2)$, where $k:x$ in cell $(i,j)$ represents that $O_k(i,j)=x$.

    \begin{figure}[h]
        \centering
        $\arraycolsep=4pt\begin{array}{|cc|cc|cc|} \hline
        \thead{1:17/2\\3:7/2} & \thead{2:5\\4:3} & 
        \thead{1:5/2\\3:1} & \thead{2:9/2\\4:2} & 
        \thead{1:4\\3:1/2} & \thead{2:11/2\\\phantom{2}}\\ \hline
        \thead{1:2\\3:1/2} & \thead{2:9/2\\4:1} & 
        \thead{1:3} & 2:1 & 
        \thead{1:1\\3:3/2\\} & \thead{2:1/2\\4:1}\\ \hline
        \thead{1:3/2} & \thead{2:5/2} & 
        \thead{1:1/2\\3:1} & \thead{2:1/2\\\phantom{2}} &
        \thead{1:1} & 4:1 \\ \hline
        \end{array}$
        \caption{A rational outline rectangle associated to $(5^12^11^1,4^12^2,3^21^2)$.}
        \label{rational outline rectangle}
    \end{figure}
\end{example}

We now add an extra property to these outline rectangles which will be useful later.

\begin{definition}
    A rational outline rectangle $O$ is \emph{symmetric} if $O_k(i,j) = O_c(a,b)$ for every permutation $(a,b,c)$ of $(i,j,k)$.
\end{definition}

With this definition, a symmetric rational outline rectangle must have $P=Q=R$ and thus $u=v=t$.

Given a rational outline rectangle, it would be helpful to be able to transform it into an outline rectangle. The following is one method for obtaining an outline rectangle from a symmetric rational outline rectangle. In all further results, note that $\{x\} = x - \lfloor x\rfloor$ denotes the fractional part of a rational number $x$.

\begin{lemma}
\label{rational outline to full}
    For a partition $P = (p_1,\dots,p_u)$ of $n$, let $O$ be a symmetric rational outline rectangle associated to $(P,P,P)$ and let $B$ be $u\times u$ array of multisets, where $y(i,j)$ denotes the number of entries in cell $(i,j)$ of $B$, such that
    \begin{itemize}
        \item $y(i,j) = \sum_{k\in[u]} \{O_k(i,j)\}$;
        \item the number of copies of symbol $j$ in row $i$ of $B$ is $y(i,j)$;
        \item and the number of copies of symbol $j$ in column $i$ of $B$ is $y(i,j)$.
    \end{itemize}
    Then there exists an outline rectangle associated to $(P,P,P)$.
\end{lemma}
\begin{proof}
    Construct a new $u\times u$ array of multisets in the symbols of $[u]$, denoted $A$, as follows. Given the symmetric rational outline rectangle $O$, let the number of copies of symbol $k$ in cell $(i,j)$ of $A$ be $A_k(i,j)$, and take $A_k(i,j) = \lfloor O_k(i,j)\rfloor$ for all $i,j,k\in[u]$. Clearly each $A_k(i,j)$ is a non-negative integer.

    Observe that cell $(i,j)$ now contains $\sum_{k\in[u]} A_k(i,j)\leq p_ip_j$ entries. Since $O_k(i,j)$ is equivalent up to permutation of $(i,j,k)$, we have that the number of copies of symbol $i$ in row $j$ or column $j$ is $\sum_{k\in[u]} A_k(i,j)\leq p_ip_j$.

    Take an array $L$ to be the cell-wise union $A\cup B$. Note that $y(i,j) = p_ip_j - \sum_{k\in[u]} A_k(i,j)$. From the conditions required of $B$, it is clear that the number of entries in cell $(i,j)$ of $L$ is $p_ip_j$. Similarly, there are $p_ip_j$ copies of symbol $j$ in row $i$ and $p_ip_j$ copies in column $i$.

    Therefore, $L$ is an outline rectangle associated to $(P,P,P)$.
\end{proof}

\begin{example}
    The array in \Cref{54321 O} shows a symmetric rational outline rectangle $O$ associated to $(P,P,P)$, where $P = (5,4,3,2,1)$, and the array in \Cref{54321 A} shows the corresponding array $A$.

    \begin{figure}[h]
    \begin{subfigure}{0.48\textwidth}
        \centering
        \renewcommand{\arraystretch}{2.7}
        \begin{tabular}{|l|l|l|l|l|}\hline
        $1:25$ & 
        \thead{$3:31/3$\\$4:19/3$\\$5:10/3$} & 
        \thead{$2:31/3$\\$4:10/3$\\$5:4/3$} & 
        \thead{$2:19/3$\\$3:10/3$\\$5:1/3$} & 
        \thead{$2:10/3$\\$3:4/3$\\$4:1/3$} \\\hline
        \thead{$3:31/3$\\$4:19/3$\\$5:10/3$} & 
        $2:16$ & 
        \thead{$1:31/3$\\$4:4/3$\\$5:1/3$} & 
        \thead{$1:19/3$\\$3:4/3$\\$5:1/3$} & 
        \thead{$1:10/3$\\$3:1/3$\\$4:1/3$} \\\hline
        \thead{$2:31/3$\\$4:10/3$\\$5:4/3$} & 
        \thead{$1:31/3$\\$4:4/3$\\$5:1/3$} & 
        $3:9$ & 
        \thead{$1:10/3$\\$2:4/3$\\$5:4/3$} & 
        \thead{$1:4/3$\\$2:1/3$\\$4:4/3$} \\\hline
        \thead{$2:19/3$\\$3:10/3$\\$5:1/3$} & 
        \thead{$1:19/3$\\$3:4/3$\\$5:1/3$} & 
        \thead{$1:10/3$\\$2:4/3$\\$5:4/3$} & 
        $4:4$ & 
        \thead{$1:1/3$\\$2:1/3$\\$3:4/3$}\\\hline
        \thead{$2:10/3$\\$3:4/3$\\$4:1/3$} &
        \thead{$1:10/3$\\$3:1/3$\\$4:1/3$} & 
        \thead{$1:4/3$\\$2:1/3$\\$4:4/3$} & 
        \thead{$1:1/3$\\$2:1/3$\\$3:4/3$} & 
        $5:1$\\\hline
        \end{tabular}
        \caption{The rational outline rectangle $O$}
        \label{54321 O}
    \end{subfigure}
    \begin{subfigure}{0.48\textwidth}
        \centering
        \renewcommand{\arraystretch}{2.7}
        $\arraycolsep=10pt\begin{array}{|l|l|l|l|l|}\hline
        1:25 & 
        \thead{$3:10$\\$4:6$\\$5:3$} & 
        \thead{$2:10$\\$4:3$\\$5:1$} & 
        \thead{$2:6$\\$3:3$\\$5:0$} & 
        \thead{$2:3$\\$3:1$\\$4:0$} \\\hline
        \thead{$3:10$\\$4:6$\\$5:3$} & 
        2:16 & 
        \thead{$1:10$\\$4:1$\\$5:0$} & 
        \thead{$1:6$\\$3:1$\\$5:0$} & 
        \thead{$1:3$\\$3:0$\\$4:0$} \\\hline
        \thead{$2:10$\\$4:3$\\$5:1$} & 
        \thead{$1:10$\\$4:1$\\$5:0$} & 
        3:9 & 
        \thead{$1:3$\\$2:1$\\$5:1$} & 
        \thead{$1:1$\\$2:0$\\$4:1$} \\\hline
        \thead{$2:6$\\$3:3$\\$5:0$} & 
        \thead{$1:6$\\$3:1$\\$5:0$} & 
        \thead{$1:3$\\$2:1$\\$5:1$} & 
        4:4 & 
        \thead{$1:0$\\$2:0$\\$3:1$}\\\hline
        \thead{$2:3$\\$3:1$\\$4:0$} &
        \thead{$1:3$\\$3:0$\\$4:0$} & 
        \thead{$1:1$\\$2:0$\\$4:1$} & 
        \thead{$1:0$\\$2:0$\\$3:1$} & 
        5:1\\\hline
        \end{array}$
        \caption{The multiset array $A$}
        \label{54321 A}
    \end{subfigure}
    \caption{Arrays $O$ and $A$ as defined in \Cref{rational outline to full} with $P = (5,4,3,2,1)$.}
    \end{figure}

    Observe that the cells on the main diagonal of $A$ have the correct number of entries, and all other cells are missing exactly one entry. Since $O$ is a symmetric rational outline rectangle, there is also one copy of each symbol missing in every row and column, except the row/column where the symbol is in the main diagonal cell. Thus, the array in \Cref{54321 B} is satisfactory for the array $B$, as defined in \Cref{rational outline to full}.

    \begin{figure}[h]
        \centering
        $$\arraycolsep=4pt\begin{array}{|c|c|c|c|c|} \hhline{|*{5}{-|}}
     & 4 & 2 & 5 & 3\\ \hhline{|*{5}{-|}}
    4 &  & 5 & 3 & 1\\ \hhline{|*{5}{-|}}
    2 & 5 &  & 1 & 4\\ \hhline{|*{5}{-|}}
    5 & 3 & 1 &  & 2\\ \hhline{|*{5}{-|}}
    3 & 1 & 4 & 2 & \\ \hhline{|*{5}{-|}}\end{array}$$
        \caption{The array $B$ as defined in \Cref{rational outline to full} with $P = (5,4,3,2,1)$.}
        \label{54321 B}
    \end{figure}
\end{example}

\begin{theorem}
\label{symROR to 2RP}
    If $O$ is a symmetric rational outline rectangle associated to $(P,P,P)$ for $P = (h_1,\dots,h_k)$, and $O_i(i,i) = h_i^2$ for all $i\in[k]$, then an outline rectangle found using \Cref{rational outline to full} has exactly $h_i^2$ copies of symbol $i$ in cell $(i,i)$ and lifts to a $\LS(h_1\dots h_k)$.
\end{theorem}

To simplify notation, we will say that the rational outline rectangle $O$ lifts to a $\LS(h_1\dots h_k)$ if $O_i(i,i) = h_i^2$ for all $i\in[k]$.

\section{Sufficient conditions for $k=5$}

Kuhl et al. provide the following conjecture in \cite{kuhl2019existence}.

\begin{conjecture}
\label{5 hole conjecture}
    A $\LS(h_1\dots h_5)$ exists if and only if $n^2 - \sum_{i=1}^5h_i^2\geq 3\left( \sum_{i\in D}h_i\right)\left(\sum_{j\in\overline{D}}h_j\right)$ for all subsets $D$ of $[5]$ where $|D| = 3$.
\end{conjecture}

Observe that the necessary and sufficient condition for existence in \Cref{5 hole conjecture} is a specific case of \Cref{squarecondition2}, so it is already shown to be necessary. In this section, we will prove that the condition is also sufficient, without relying on any previous results on the existence of realizations.

\begin{lemma}
\label{rational outline ALL 5}
    For $P = (h_1,\dots,h_5)$, if $n^2 - \sum_{i=1}^5h_i^2\geq 3\left( \sum_{i\in D}h_i\right)\left(\sum_{j\in\overline{D}}h_j\right)$ for all subsets $D$ of $[5]$ where $|D| = 3$, then there exists a rational outline rectangle which lifts to a $\LS(h_1\dots h_5)$.
\end{lemma}
\begin{proof}
    Suppose that $n^2 - \sum_{i=1}^5h_i^2\geq 3\left( \sum_{i\in D}h_i\right)\left(\sum_{j\in\overline{D}}h_j\right)$ for all subsets $D$ of $[5]$ where $|D| = 3$.
    
    Given the partition $(h_1\dots h_5)$, let $O$ be an empty $5\times 5$ array.
    Denote the number of copies of symbol $k$ in cell $(i,j)$ of $O$ by $x_k(i,j)$. Then let $$x_k(i,j) = \begin{cases}
        h_i^2, & \text{if $i=j=k$,}\\
        0, & \text{if $i=j\neq k$, $i=k\neq j$ or $i\neq j=k$,}\\
        \frac{1}{6}\left(n^2 - \sum_{m\in[5]}h_m^2 - 3(h_i+h_j+h_k)(n-h_i-h_j-h_k)\right), & \text{otherwise.}
    \end{cases}$$

    Clearly the cells $(i,i)$ for all $i\in[5]$ have only $h_i^2$ copies of symbol $i$. Thus, if $O$ is a valid rational outline rectangle, it will lift to a $\LS(h_1\dots h_5)$.

    For the cell $(i,j)$, where $i\neq j$, the number of entries in the cell is
    $$\sum_{k\in[5]\setminus\{i,j\}}x_k(i,j) = \frac{3}{6}\left(n^2 - \sum_{k\in[5]}h_k^2 - \sum_{k\in[5]\setminus\{i,j\}}(h_i+h_j+h_k)(n-h_i-h_j-h_k)\right)$$
    and
    \begin{align*}
        \sum_{k\in[5]\setminus\{i,j\}}(h_i+h_j+h_k)(n-h_i-h_j-h_k) &= n\sum_{k\in[5]\setminus\{i,j\}}(h_i+h_j+h_k) - \sum_{k\in[5]\setminus\{i,j\}}(h_i+h_j+h_k)^2\\
        &= n(2h_i+2h_j+n) - \sum_{k\in[5]\setminus\{i,j\}}(h_i^2+h_j^2+h_k^2+(2h_i+2h_j)h_k+2h_ih_j)\\
        &= n(2h_i+2h_j+n) - 2h_ih_j - \sum_{k\in[5]}(h_k^2+(2h_i+2h_j)h_k)\\
        &= n(2h_i+2h_j+n) - 2h_ih_j - n(2h_i+2h_j) - \sum_{k\in[5]}h_k^2\\
        &= n^2 - \sum_{k\in[5]}h_k^2 - 2h_ih_j
    \end{align*}
    Thus, the number of entries in cell $(i,j)$ is
    $$\sum_{k\in[5]\setminus\{i,j\}}x_k(i,j) = \frac{1}{2}(n^2 - \sum_{k\in[5]}h_k^2 - n^2 + \sum_{k\in[5]}h_k^2 + 2h_ih_j) = h_ih_j.$$

    Observe that $x_k(i,j) = x_k(j,i) = x_i(j,k) = x_i(k,j) = x_j(i,k) = x_j(k,i)$, so the number of times that symbol $j$ appears in row $i$ is
    $$\sum_{k\in[5]\setminus\{i,j\}}x_j(i,k) = \sum_{k\in[5]\setminus\{i,j\}}x_k(i,j) = h_ih_j$$
    and the number of times that $j$ appears in column $i$ is
    $$\sum_{k\in[5]\setminus\{i,j\}}x_j(k,i) = \sum_{k\in[5]\setminus\{i,j\}}x_k(i,j) = h_ih_j.$$

    Finally, we check that all of the $x_k(i,j)$ are non-negative. With $D = \{i,j,k\}$, our initial assumption gives that $$n^2 - \sum_{m\in[5]}h_m^2 \geq 3(h_i+h_j+h_k)(n-h_i-h_j-h_k).$$
    It follows that $$\frac{1}{6}\left(n^2 - \sum_{m\in[5]}h_m^2 - 3(h_i+h_j+h_k)(n-h_i-h_j-h_k)\right)\geq 0.$$

    Therefore, for any partition $P=(h_1\dots h_5)$ which satisfies the given condition, $O$ is a rational outline rectangle associated to $(P,P,P)$ which lifts to a $\LS(h_1\dots h_5)$.
\end{proof}

\begin{theorem}
\label{All 5 parts}
    For any partition $P = (h_1\dots h_5)$, a $\LS(h_1\dots h_5)$ exists if and only if $n^2 - \sum_{i=1}^5h_i^2\geq 3\left( \sum_{i\in D}h_i\right)\left(\sum_{j\in\overline{D}}h_j\right)$ for all subsets $D$ of $[5]$ where $|D| = 3$.
\end{theorem}
\begin{proof}
    As mentioned earlier, if a $\LS(h_1\dots h_5)$ exists then \Cref{squarecondition2} gives that $n^2 - \sum_{i=1}^5h_i^2\geq 3\left( \sum_{i\in D}h_i\right)\left(\sum_{j\in\overline{D}}h_j\right)$ for all subsets $D$ of $[5]$ where $|D| = 3$. Thus the forward direction is proven already.

    We prove the existence of a $\LS(h_1\dots h_5)$ by resolving the symmetric rational outline rectangle constructed in \Cref{rational outline ALL 5} as in \Cref{rational outline to full}. Let $A$ be a $5\times 5$ array of multisets and denote the number of copies of symbol $k$ in cell $(i,j)$ by $A_k(i,j)$. Now let $A_k(i,j) = \lfloor x_k(i,j)\rfloor$, where $x_k(i,j)$ is as defined in \Cref{rational outline ALL 5}. So $$A_k(i,j) = \begin{cases}
        h_i^2, & \text{if $i=j=k$,}\\
        0, & \text{if $i=j\neq k$, $i=k\neq j$ or $i\neq j=k$,}\\
        \left\lfloor\frac{1}{6}\left(n^2 - \sum_{m\in[5]}h_m^2 - 3(h_i+h_j+h_k)(n-h_i-h_j-h_k)\right)\right\rfloor, & \text{otherwise.}
    \end{cases}$$

    Using \Cref{rational outline to full}, we now only need to construct the array $B$ in order to have a complete outline rectangle.

    Let $y(i,j)$ be the number of entries needed in cell $(i,j)$. Thus $$y(i,j) = \begin{cases}
        0, & \text{if $i=j$,}\\
        \sum_{k\in[5]\setminus\{i,j\}}\{x_k(i,j)\}, & \text{otherwise.}
    \end{cases}$$
    Note that $y(i,j)$ also represents the number of copies of $j$ needed in row $i$ or column $i$, that $y(i,j) = y(j,i)$ and that $0\leq y(i,j)\leq 2$. Thus, we only need to consider the $\binom{5}{2} = 10$ pairs where $i<j$. Similarly, we are only concerned about the values of $x_k(i,j)$ where $i\neq j\neq k$, and since $x_k(i,j) = x_k(j,i) = x_i(j,k) = x_i(k,j) = x_j(i,k) = x_j(k,i)$, the order of $i$, $j$ and $k$ is not important. Also, we are working with $\{x_k(i,j)\}$, so let $x(i,j,k) = \{x_{k'}(i',j')\}$ for any $\{i',j',k'\} = \{i,j,k\}\subset[5]$, where $i<j<k$.

    To aid in understanding the possible values of $y(i,j)$, we will visualise the $y(i,j)$ values as edges of a graph. We use a vertex for each $i\in[5]$, and join two vertices $i$ and $j$ by an edge with a style corresponding to the value of $y(i,j)$. We use dashed edges for $0$, solid edges for $1$ and dotted edges for $2$. The graph will be a complete graph on 5 points.

    We aim to show that there are only 9 possible graphs, up to permutation of the vertices. These graphs are shown in \Cref{Graph options}.

\begin{figure}[h]
    \centering
    \begin{tikzpicture}
        \def\ngon{5}
        \node[regular polygon,regular polygon sides=\ngon,minimum size=2cm] (p) {};
        \foreach\x in {1,...,\ngon}{\node[draw, circle, fill=black, minimum size=0.2cm, inner sep=0pt] (p\x) at (p.corner \x){};}

        \draw[dashed] (p1) -- (p2);
        \draw[dashed] (p1) -- (p3);
        \draw[dashed] (p1) -- (p4);
        \draw[dashed] (p1) -- (p5);
        \draw[dashed] (p2) -- (p3);
        \draw[dashed] (p2) -- (p4);
        \draw[dashed] (p2) -- (p5);
        \draw[dashed] (p3) -- (p4);
        \draw[dashed] (p3) -- (p5);
        \draw[dashed] (p4) -- (p5);
    \end{tikzpicture}
    \quad
    \begin{tikzpicture}
        \def\ngon{5}
        \node[regular polygon,regular polygon sides=\ngon,minimum size=2cm] (p) {};
        \foreach\x in {1,...,\ngon}{\node[draw, circle, fill=black, minimum size=0.2cm, inner sep=0pt] (p\x) at (p.corner \x){};}

        \draw[dashed] (p1) -- (p2);
        \draw[dashed] (p1) -- (p3);
        \draw[dashed] (p1) -- (p4);
        \draw[dashed] (p1) -- (p5);
        \draw (p2) -- (p3);
        \draw (p2) -- (p4);
        \draw (p2) -- (p5);
        \draw (p3) -- (p4);
        \draw (p3) -- (p5);
        \draw (p4) -- (p5);
    \end{tikzpicture}
    \quad
    \begin{tikzpicture}
        \def\ngon{5}
        \node[regular polygon,regular polygon sides=\ngon,minimum size=2cm] (p) {};
        \foreach\x in {1,...,\ngon}{\node[draw, circle, fill=black, minimum size=0.2cm, inner sep=0pt] (p\x) at (p.corner \x){};}

        \draw[dashed] (p1) -- (p2);
        \draw (p1) -- (p3);
        \draw (p1) -- (p4);
        \draw (p1) -- (p5);
        \draw (p2) -- (p3);
        \draw (p2) -- (p4);
        \draw (p2) -- (p5);
        \draw (p3) -- (p4);
        \draw (p3) -- (p5);
        \draw (p4) -- (p5);
    \end{tikzpicture}
    \quad
    \begin{tikzpicture}
        \def\ngon{5}
        \node[regular polygon,regular polygon sides=\ngon,minimum size=2cm] (p) {};
        \foreach\x in {1,...,\ngon}{\node[draw, circle, fill=black, minimum size=0.2cm, inner sep=0pt] (p\x) at (p.corner \x){};}

        \draw (p1) -- (p2);
        \draw (p1) -- (p3);
        \draw (p1) -- (p4);
        \draw (p1) -- (p5);
        \draw (p2) -- (p3);
        \draw (p2) -- (p4);
        \draw (p2) -- (p5);
        \draw (p3) -- (p4);
        \draw (p3) -- (p5);
        \draw (p4) -- (p5);
    \end{tikzpicture}
    \begin{tikzpicture}
        \def\ngon{5}
        \node[regular polygon,regular polygon sides=\ngon,minimum size=2cm] (p) {};
        \foreach\x in {1,...,\ngon}{\node[draw, circle, fill=black, minimum size=0.2cm, inner sep=0pt] (p\x) at (p.corner \x){};}

        \draw[densely dotted] (p1) -- (p2);
        \draw (p1) -- (p3);
        \draw (p1) -- (p4);
        \draw (p1) -- (p5);
        \draw (p2) -- (p3);
        \draw (p2) -- (p4);
        \draw (p2) -- (p5);
        \draw (p3) -- (p4);
        \draw (p3) -- (p5);
        \draw (p4) -- (p5);
    \end{tikzpicture}
    \quad
    \begin{tikzpicture}
        \def\ngon{5}
        \node[regular polygon,regular polygon sides=\ngon,minimum size=2cm] (p) {};
        \foreach\x in {1,...,\ngon}{\node[draw, circle, fill=black, minimum size=0.2cm, inner sep=0pt] (p\x) at (p.corner \x){};}

        \draw (p1) -- (p2);
        \draw (p1) -- (p3);
        \draw (p1) -- (p4);
        \draw (p1) -- (p5);
        \draw[densely dotted] (p2) -- (p3);
        \draw[densely dotted] (p2) -- (p4);
        \draw[densely dotted] (p2) -- (p5);
        \draw[densely dotted] (p3) -- (p4);
        \draw[densely dotted] (p3) -- (p5);
        \draw[densely dotted] (p4) -- (p5);
    \end{tikzpicture}
    \quad
    \begin{tikzpicture}
        \def\ngon{5}
        \node[regular polygon,regular polygon sides=\ngon,minimum size=2cm] (p) {};
        \foreach\x in {1,...,\ngon}{\node[draw, circle, fill=black, minimum size=0.2cm, inner sep=0pt] (p\x) at (p.corner \x){};}

        \draw[densely dotted] (p1) -- (p2);
        \draw[densely dotted] (p1) -- (p3);
        \draw[densely dotted] (p1) -- (p4);
        \draw[densely dotted] (p1) -- (p5);
        \draw[densely dotted] (p2) -- (p3);
        \draw[densely dotted] (p2) -- (p4);
        \draw[densely dotted] (p2) -- (p5);
        \draw[densely dotted] (p3) -- (p4);
        \draw[densely dotted] (p3) -- (p5);
        \draw[densely dotted] (p4) -- (p5);
    \end{tikzpicture}
    \quad
    \begin{tikzpicture}
        \def\ngon{5}
        \node[regular polygon,regular polygon sides=\ngon,minimum size=2cm] (p) {};
        \foreach\x in {1,...,\ngon}{\node[draw, circle, fill=black, minimum size=0.2cm, inner sep=0pt] (p\x) at (p.corner \x){};}

        \draw (p1) -- (p2);
        \draw[densely dotted] (p1) -- (p3);
        \draw[densely dotted] (p1) -- (p4);
        \draw[densely dotted] (p1) -- (p5);
        \draw[densely dotted] (p2) -- (p3);
        \draw[densely dotted] (p2) -- (p4);
        \draw[densely dotted] (p2) -- (p5);
        \draw[densely dotted] (p3) -- (p4);
        \draw[densely dotted] (p3) -- (p5);
        \draw[densely dotted] (p4) -- (p5);
    \end{tikzpicture}
    \quad
    \begin{tikzpicture}
        \def\ngon{5}
        \node[regular polygon,regular polygon sides=\ngon,minimum size=2cm] (p) {};
        \foreach\x in {1,...,\ngon}{\node[draw, circle, fill=black, minimum size=0.2cm, inner sep=0pt] (p\x) at (p.corner \x){};}

        \draw[densely dotted] (p1) -- (p2);
        \draw[densely dotted] (p1) -- (p3);
        \draw[densely dotted] (p1) -- (p4);
        \draw[densely dotted] (p1) -- (p5);
        \draw (p2) -- (p3);
        \draw (p2) -- (p4);
        \draw (p2) -- (p5);
        \draw (p3) -- (p4);
        \draw (p3) -- (p5);
        \draw (p4) -- (p5);
    \end{tikzpicture}

    \vspace{0.5cm}
    \begin{tikzpicture}
        \draw[dashed] (0,0) -- (1,0);
        \node at (1.25,0) {0};
        \draw (2,0) -- (3,0);
        \node at (3.25,0) {1};
        \draw[densely dotted] (4,0) -- (5,0);
        \node at (5.25,0) {2};
    \end{tikzpicture}
    \caption{The nine possible graphs representing the values of $y(i,j)$.}
    \label{Graph options}
\end{figure}
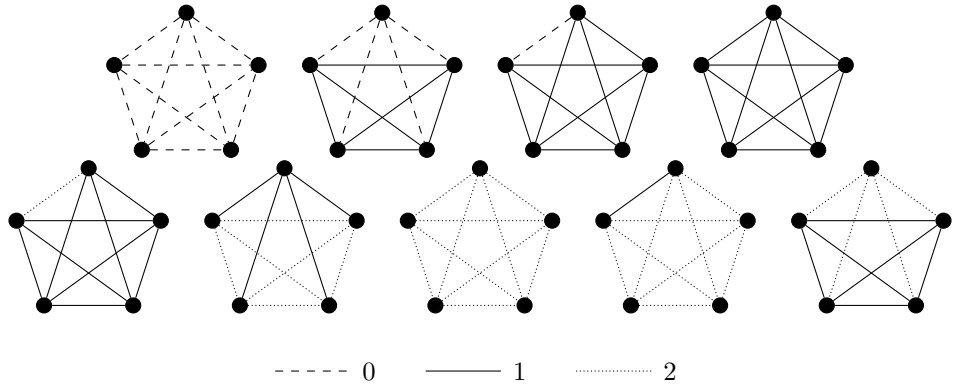

    To prove that these are the only possible graphs, we will start by assuming that at least one edge is 0.

    \noindent\textbf{At least one edge labelled 0}

    In this case, we now show that any such graph must satisfy the following property.

    \begin{itemize}
        \item If $y(i,j)=0$, then either $y(i,k) = 0$ for all $k\in[5]\setminus\{i,j\}$ or $y(i,k) = 1$ for all $k\in[5]\setminus\{i,j\}$.
    \end{itemize}

    Suppose, without loss of generality, that $y(1,2) = 0$. This implies that $x(1,2,k) = 0$ for all $k\in\{3,4,5\}$. Now consider the values of $y(1,j)$ for $j\in\{3,4,5\}$.
    \begin{align*}
        y(1,3) &= x(1,2,3) + x(1,3,4) + x(1,3,5) = x(1,3,4) + x(1,3,5)\\
        y(1,4) &= x(1,2,4) + x(1,3,4) + x(1,4,5) = x(1,3,4) + x(1,4,5)\\
        y(1,5) &= x(1,2,5) + x(1,3,5) + x(1,4,5) = x(1,3,5) + x(1,4,5)
    \end{align*}
    These three values are dependent on combinations of just three $x$ values. Also $y(1,j) \leq 1$, and if $y(1,j)=1$ then both values of $x(i,j,k)$ must be positive. Thus, if $y(1,j) = 0$ for some $j\in\{3,4,5\}$, it is not possible for the other two $y$ values to be 1. Therefore, if $y(i,j)=0$, then either $y(i,k) = 0$ for all $k\in[5]\setminus\{i,j\}$ or $y(i,k) = 1$ for all $k\in[5]\setminus\{i,j\}$.

    We will now find all possible graphs with at least one dashed (0) edge. We begin by assuming that $y(i,j)=0$ for some $i,j\in[5]$ where $i\neq j$. So we know that there must be a dashed edge, and, from property (1), for any vertex with a dashed edge, the remaining edges are either all 0 or all 1. Thus, each vertex incident with a dashed edge has two options, as shown in \Cref{0/1 options}.

\begin{figure}[h]
    \centering
    \begin{tikzpicture}
        \def\ngon{5}
        \node[regular polygon,regular polygon sides=\ngon,minimum size=2cm] (p) {};
        \foreach\x in {1,...,\ngon}{\node[draw, circle, fill=black, minimum size=0.2cm, inner sep=0pt] (p\x) at (p.corner \x){};}

        \draw[dashed] (p1) -- (p2);
        \draw[dashed] (p1) -- (p3);
        \draw[dashed] (p1) -- (p4);
        \draw[dashed] (p1) -- (p5);
    \end{tikzpicture}
    \quad
    \begin{tikzpicture}
        \def\ngon{5}
        \node[regular polygon,regular polygon sides=\ngon,minimum size=2cm] (p) {};
        \foreach\x in {1,...,\ngon}{\node[draw, circle, fill=black, minimum size=0.2cm, inner sep=0pt] (p\x) at (p.corner \x){};}

        \draw[dashed] (p1) -- (p2);
        \draw (p1) -- (p3);
        \draw (p1) -- (p4);
        \draw (p1) -- (p5);
    \end{tikzpicture}
    \caption{The options for a single vertex.}
    \label{0/1 options}
\end{figure}
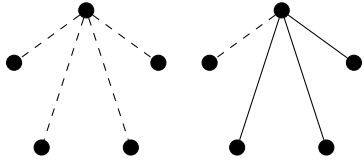

    Given these conditions, suppose that some $y(i,j)=0$, so there is a dashed line in the graph. Using the two states from \Cref{0/1 options}, each state then has two scenarios for the leftmost vertex.

\begin{figure}[h]
    \centering
    \begin{tikzpicture}
        \def\ngon{5}
        \node[regular polygon,regular polygon sides=\ngon,minimum size=2cm] (p) {};
        \foreach\x in {1,...,\ngon}{\node[draw, circle, fill=black, minimum size=0.2cm, inner sep=0pt] (p\x) at (p.corner \x){};}

        \draw[dashed] (p1) -- (p2);
        \draw[dashed] (p1) -- (p3);
        \draw[dashed] (p1) -- (p4);
        \draw[dashed] (p1) -- (p5);
        \draw[dashed] (p2) -- (p3);
        \draw[dashed] (p2) -- (p4);
        \draw[dashed] (p2) -- (p5);
    \end{tikzpicture}
    \quad
    \begin{tikzpicture}
        \def\ngon{5}
        \node[regular polygon,regular polygon sides=\ngon,minimum size=2cm] (p) {};
        \foreach\x in {1,...,\ngon}{\node[draw, circle, fill=black, minimum size=0.2cm, inner sep=0pt] (p\x) at (p.corner \x){};}

        \draw[dashed] (p1) -- (p2);
        \draw[dashed] (p1) -- (p3);
        \draw[dashed] (p1) -- (p4);
        \draw[dashed] (p1) -- (p5);
        \draw (p2) -- (p3);
        \draw (p2) -- (p4);
        \draw (p2) -- (p5);
    \end{tikzpicture}
    \quad
    \begin{tikzpicture}
        \def\ngon{5}
        \node[regular polygon,regular polygon sides=\ngon,minimum size=2cm] (p) {};
        \foreach\x in {1,...,\ngon}{\node[draw, circle, fill=black, minimum size=0.2cm, inner sep=0pt] (p\x) at (p.corner \x){};}

        \draw[dashed] (p1) -- (p2);
        \draw (p1) -- (p3);
        \draw (p1) -- (p4);
        \draw (p1) -- (p5);
        \draw[dashed] (p2) -- (p3);
        \draw[dashed] (p2) -- (p4);
        \draw[dashed] (p2) -- (p5);
    \end{tikzpicture}
    \quad
    \begin{tikzpicture}
        \def\ngon{5}
        \node[regular polygon,regular polygon sides=\ngon,minimum size=2cm] (p) {};
        \foreach\x in {1,...,\ngon}{\node[draw, circle, fill=black, minimum size=0.2cm, inner sep=0pt] (p\x) at (p.corner \x){};}

        \draw[dashed] (p1) -- (p2);
        \draw (p1) -- (p3);
        \draw (p1) -- (p4);
        \draw (p1) -- (p5);
        \draw (p2) -- (p3);
        \draw (p2) -- (p4);
        \draw (p2) -- (p5);
    \end{tikzpicture}
    \caption{The updated scenarios with the edges added for a second vertex.}
    \label{0/1 options 2}
\end{figure}
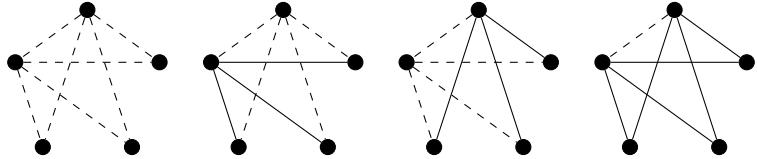

    The remaining edges of the first three graphs are then forced.

\begin{figure}[h]
    \centering
    \begin{tikzpicture}
        \def\ngon{5}
        \node[regular polygon,regular polygon sides=\ngon,minimum size=2cm] (p) {};
        \foreach\x in {1,...,\ngon}{\node[draw, circle, fill=black, minimum size=0.2cm, inner sep=0pt] (p\x) at (p.corner \x){};}

        \draw[dashed] (p1) -- (p2);
        \draw[dashed] (p1) -- (p3);
        \draw[dashed] (p1) -- (p4);
        \draw[dashed] (p1) -- (p5);
        \draw[dashed] (p2) -- (p3);
        \draw[dashed] (p2) -- (p4);
        \draw[dashed] (p2) -- (p5);
        \draw[dashed] (p3) -- (p4);
        \draw[dashed] (p3) -- (p5);
        \draw[dashed] (p4) -- (p5);
    \end{tikzpicture}
    \quad
    \begin{tikzpicture}
        \def\ngon{5}
        \node[regular polygon,regular polygon sides=\ngon,minimum size=2cm] (p) {};
        \foreach\x in {1,...,\ngon}{\node[draw, circle, fill=black, minimum size=0.2cm, inner sep=0pt] (p\x) at (p.corner \x){};}

        \draw[dashed] (p1) -- (p2);
        \draw[dashed] (p1) -- (p3);
        \draw[dashed] (p1) -- (p4);
        \draw[dashed] (p1) -- (p5);
        \draw (p2) -- (p3);
        \draw (p2) -- (p4);
        \draw (p2) -- (p5);
        \draw (p3) -- (p4);
        \draw (p3) -- (p5);
        \draw (p4) -- (p5);
    \end{tikzpicture}
    \quad
    \begin{tikzpicture}
        \def\ngon{5}
        \node[regular polygon,regular polygon sides=\ngon,minimum size=2cm] (p) {};
        \foreach\x in {1,...,\ngon}{\node[draw, circle, fill=black, minimum size=0.2cm, inner sep=0pt] (p\x) at (p.corner \x){};}

        \draw[dashed] (p1) -- (p2);
        \draw (p1) -- (p3);
        \draw (p1) -- (p4);
        \draw (p1) -- (p5);
        \draw[dashed] (p2) -- (p3);
        \draw[dashed] (p2) -- (p4);
        \draw[dashed] (p2) -- (p5);
        \draw (p3) -- (p4);
        \draw (p3) -- (p5);
        \draw (p4) -- (p5);
    \end{tikzpicture}
    \caption{Three possible graphs with an edge labelled 0.}
    \label{0/1 options 3}
\end{figure}
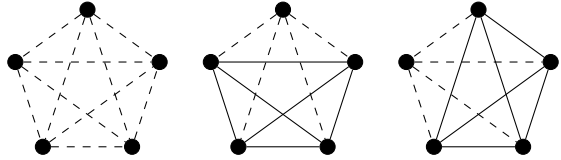

    For the final graph, suppose that $y(1,2)=0$ is the dashed edge. Then as above, we know that
    \begin{align*}
        1 = y(1,3) &= x(1,3,4) + x(1,3,5)\\
        1 = y(1,4) &= x(1,3,4) + x(1,4,5)\\
        1 = y(1,5) &= x(1,3,5) + x(1,4,5)
    \end{align*}
    It follows from the first two equations that $x(1,3,5) = x(1,4,5)$, and so $x(1,3,4) = x(1,3,5) = x(1,4,5) = 0.5$. Repeating this for $1=y(2,3)=y(2,4)=y(2,5)$, we know that $x(2,3,4) = x(2,3,5) = x(2,4,5) = 0.5$.

    The empty edges of the graph correspond to $y(3,4)$, $y(3,5)$ and $y(4,5)$. With what was found above, we have that
    \begin{align*}
        y(3,4) &= 0.5 + 0.5 + x(3,4,5) = 1 + x(3,4,5)\\
        y(3,5) &= 0.5 + 0.5 + x(3,4,5) = 1 + x(3,4,5)\\
        y(4,5) &= 0.5 + 0.5 + x(3,4,5) = 1 + x(3,4,5)
    \end{align*}
    Since $x(3,4,5) < 1$, it must be that $y(3,4) = y(3,5) = y(4,5) = 1$.

    Thus, the only valid solution to the final graph is shown in \Cref{0/1 options 4}.

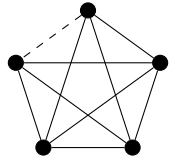
\begin{figure}[h]
    \centering
    \begin{tikzpicture}
        \def\ngon{5}
        \node[regular polygon,regular polygon sides=\ngon,minimum size=2cm] (p) {};
        \foreach\x in {1,...,\ngon}{\node[draw, circle, fill=black, minimum size=0.2cm, inner sep=0pt] (p\x) at (p.corner \x){};}

        \draw[dashed] (p1) -- (p2);
        \draw (p1) -- (p3);
        \draw (p1) -- (p4);
        \draw (p1) -- (p5);
        \draw (p2) -- (p3);
        \draw (p2) -- (p4);
        \draw (p2) -- (p5);
        \draw (p3) -- (p4);
        \draw (p3) -- (p5);
        \draw (p4) -- (p5);
    \end{tikzpicture}
    \caption{The final graph with an edge labelled 0.}
    \label{0/1 options 4}
\end{figure}

    We have shown that there are three unique graphs with at least one dashed (0) edge up to permutation of the vertices, and these are the first three graphs of the nine in \Cref{Graph options}.
    
    Observe that none of these graphs include a dotted line (2), so there cannot be both $y(i,j) = 0$ and $y(a,b) = 2$ at the same time.

    We must now determine the possible graphs when $y(i,j)\in\{1,2\}$.

    \noindent\textbf{All edges labelled 1 or 2}

    First, we will state a property that the graph must satisfy.

    \begin{itemize}
        \item \textit{Rule 1:} For any $i\in[5]$, the set of values $\{y(i,k) \mid k\in[5]\setminus\{i\}\}$ cannot be $\{1,1,2,2\}$.
    \end{itemize}

    Suppose, without loss of generality, that $y(1,2) = y(1,3) = 2$. Then since $x(1,2,3)<1$, we must have $x(1,3,4) + x(1,3,5) > 1$ and $x(1,2,4) + x(1,2,5)>1$. Now, $y(1,4) + y(1,5) = x(1,2,4) + x(1,2,5) + x(1,3,4) + x(1,3,5) + 2x(1,4,5) > 2$. Thus, at least one of $y(1,4)$ and $y(1,5)$ must be 2. Therefore, for any $i\in[5]$, the set of values $\{y(i,k) \mid k\in[5]\setminus\{i\}\}$ cannot be $\{1,1,2,2\}$. This means that in such a set, there are at least 3 copies of either 1 or 2.

    The next property we require is:

    \begin{itemize}
        \item \textit{Rule 2:} The subgraphs in \Cref{Bad subgraphs} are not possible.
    \end{itemize}

\begin{figure}[h]
    \centering
    \begin{tikzpicture}
        \def\ngon{5}
        \node[regular polygon,regular polygon sides=\ngon,minimum size=2cm] (p) {};
        \foreach\x in {1,...,\ngon}{\node[draw, circle, fill=black, minimum size=0.2cm, inner sep=0pt] (p\x) at (p.corner \x){};}

        \draw[densely dotted] (p1) -- (p2);
        \draw (p1) -- (p3);
        \draw (p1) -- (p4);
        \draw (p1) -- (p5);
        \draw[densely dotted] (p3) -- (p4);
        \draw[densely dotted] (p3) -- (p5);
        \draw[densely dotted] (p4) -- (p5);
    \end{tikzpicture}
    \quad
    \begin{tikzpicture}
        \def\ngon{5}
        \node[regular polygon,regular polygon sides=\ngon,minimum size=2cm] (p) {};
        \foreach\x in {1,...,\ngon}{\node[draw, circle, fill=black, minimum size=0.2cm, inner sep=0pt] (p\x) at (p.corner \x){};}

        \draw (p1) -- (p2);
        \draw[densely dotted] (p1) -- (p3);
        \draw[densely dotted] (p1) -- (p4);
        \draw[densely dotted] (p1) -- (p5);
        \draw (p3) -- (p4);
        \draw (p3) -- (p5);
        \draw (p4) -- (p5);
    \end{tikzpicture}
    \quad
    \begin{tikzpicture}
        \def\ngon{5}
        \node[regular polygon,regular polygon sides=\ngon,minimum size=2cm] (p) {};
        \foreach\x in {1,...,\ngon}{\node[draw, circle, fill=black, minimum size=0.2cm, inner sep=0pt] (p\x) at (p.corner \x){};}

        \draw[densely dotted] (p1) -- (p2);
        \draw (p1) -- (p5);
        \draw (p2) -- (p3);
        \draw[densely dotted] (p3) -- (p5);
    \end{tikzpicture}
    \caption{Three edge patterns which are impossible.}
    \label{Bad subgraphs}
\end{figure}
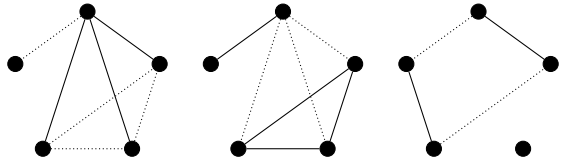

    Consider the set of values $\{y(i,k) \mid k\in[5]\setminus\{i\}\}$. Without loss of generality, let $i=1$. Suppose also that $y(1,2) = 1$ and $y(1,3)=y(1,4)=y(1,5)=2$. Then $$5 = y(1,3)+y(1,4)+y(1,5)-y(1,2) = 2(x(1,3,4)+x(1,3,5)+x(1,4,5)).$$ Since $y(2,3)+y(2,4)+y(2,5)-y(1,2)\geq2$, we get that $2(x(2,3,4)+x(2,3,5)+x(2,4,5))\geq 2$. We can combine this to get that 
    \begin{align*}
        2(y(3,4)+y(3,5)+y(4,5)) &= 2(x(1,3,4)+x(1,3,5)+x(1,4,5)) +2(x(2,3,4)+\\ &\quad\;x(2,3,5)+x(2,4,5)) + 6x(3,4,5)\\
        &\geq7.
    \end{align*} This implies that $y(3,4)+y(3,5)+y(4,5)>3$, so at least one of these values is 2.

    Suppose instead that $y(1,2) = 2$ and $y(1,3)=y(1,4)=y(1,5)=1$. Then $$1 = y(1,3)+y(1,4)+y(1,5)-y(1,2) = 2(x(1,3,4)+x(1,3,5)+x(1,4,5)).$$ Since $y(2,3)+y(2,4)+y(2,5)-y(1,2)\leq4$, we get that $2(x(2,3,4)+x(2,3,5)+x(2,4,5))\leq4$. We can combine this to get
    \begin{align*}
        2(y(3,4)+y(3,5)+y(4,5)) &= 2(x(1,3,4)+x(1,3,5)+x(1,4,5)) +2(x(2,3,4)+\\ &\quad\;x(2,3,5)+x(2,4,5)) + 6x(3,4,5)\\
        &\leq 5 + 6x(3,4,5)\\
        &< 11.
    \end{align*} Implying $y(3,4)+y(3,5)+y(4,5)<6$. So at least one of these values is 1.

    These restrictions show that the first two graphs of \Cref{Bad subgraphs} are not possible.

    Suppose that $y(i,j) = y(k,l) = 1$ and $y(i,k) = y(j,l) = 2$ for $i\neq j\neq k\neq l$. Without loss of generality, let $(i,j)=(1,2)$ and $(k,l)=(3,4)$. Since $y(1,3) = 2$, we know that $x(1,2,3) + x(1,3,4) > 1$. Also, $2 = y(1,2)+y(3,4) = x(1,2,3) + x(1,3,4) + x(1,2,4) + x(2,3,4) + x(1,2,5) + x(3,4,5)$, so $x(1,2,4) + x(2,3,4) < 1$. However, this implies that $y(2,4) = x(1,2,4) + x(2,3,4) + x(2,4,5) = 1$. Therefore, it is not possible to have $y(i,j) = y(k,l) = 1$ and $y(i,k) = y(j,l) = 2$ for $i\neq j\neq k\neq l$.

    This implies that the last graph in \Cref{Bad subgraphs} is not possible, which completes Rule 2.

    We now need only dotted and solid edges, and we will consider the assignment to 1 and 2 to be interchangeable. Rule 1 specifies that at each vertex there are two options, without considering whether solid edges are 1 or 2. \Cref{1/2 options tree} shows the possible graphs that can be constructed from these two options. The second row shows the three possible options for each of the two graphs when considering the leftmost vertex and Rule 1. To complete these graphs, it is necessary to consider both of Rules 1 and 2. Of these six graphs, the first graph has 2 solutions (up to permutation of the vertices), the fifth graph has none and the rest have a unique solution, so we are left with the six graphs given in the third row.

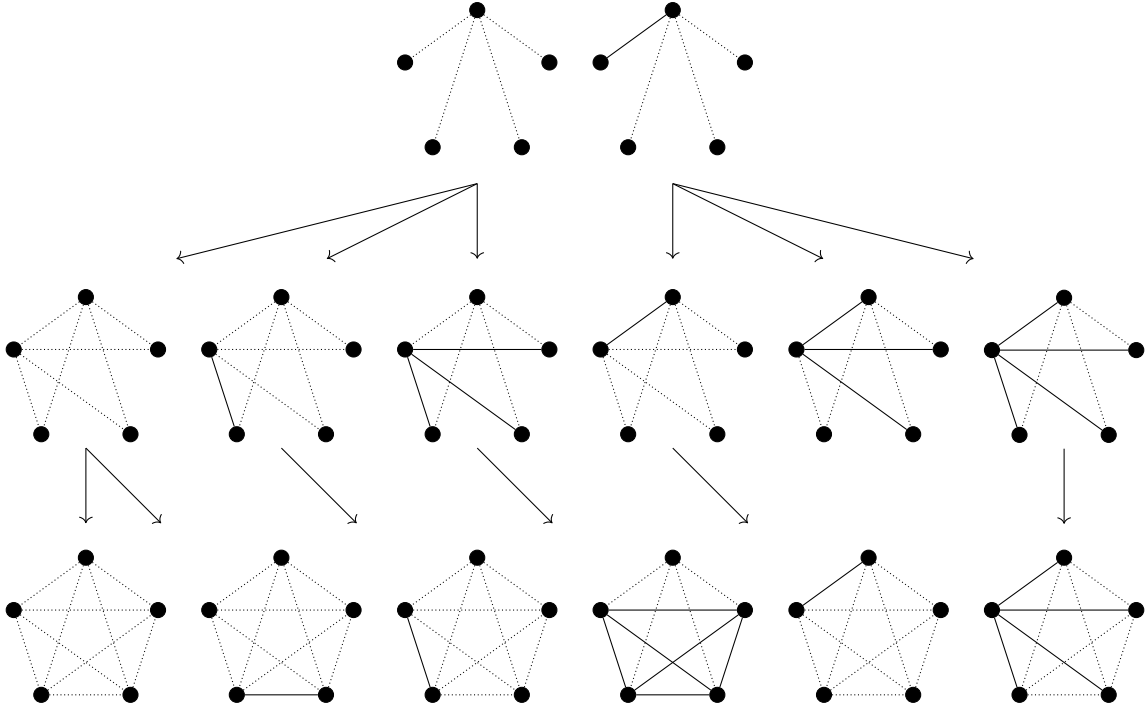
\begin{figure}[h]
    \centering

    \begin{tikzpicture}
        \def\ngon{5}
        \node[regular polygon,regular polygon sides=\ngon,minimum size=2cm] (p) {};
        \foreach\x in {1,...,\ngon}{\node[draw, circle, fill=black, minimum size=0.2cm, inner sep=0pt] (p\x) at (p.corner \x){};}

        \draw[densely dotted] (p1) -- (p2);
        \draw[densely dotted] (p1) -- (p3);
        \draw[densely dotted] (p1) -- (p4);
        \draw[densely dotted] (p1) -- (p5);

        \draw[->] (0,-1.3) -- (0,-2.3);
        \draw[->] (0,-1.3) -- (-2,-2.3);
        \draw[->] (0,-1.3) -- (-4,-2.3);
    \end{tikzpicture}
    \quad
    \begin{tikzpicture}
        \def\ngon{5}
        \node[regular polygon,regular polygon sides=\ngon,minimum size=2cm] (p) {};
        \foreach\x in {1,...,\ngon}{\node[draw, circle, fill=black, minimum size=0.2cm, inner sep=0pt] (p\x) at (p.corner \x){};}

        \draw (p1) -- (p2);
        \draw[densely dotted] (p1) -- (p3);
        \draw[densely dotted] (p1) -- (p4);
        \draw[densely dotted] (p1) -- (p5);

        \draw[->] (0,-1.3) -- (0,-2.3);
        \draw[->] (0,-1.3) -- (2,-2.3);
        \draw[->] (0,-1.3) -- (4,-2.3);
    \end{tikzpicture}

    \vspace{0.3cm}

    \begin{tikzpicture}
        \def\ngon{5}
        \node[regular polygon,regular polygon sides=\ngon,minimum size=2cm] (p) {};
        \foreach\x in {1,...,\ngon}{\node[draw, circle, fill=black, minimum size=0.2cm, inner sep=0pt] (p\x) at (p.corner \x){};}

        \draw[densely dotted] (p1) -- (p2);
        \draw[densely dotted] (p1) -- (p3);
        \draw[densely dotted] (p1) -- (p4);
        \draw[densely dotted] (p1) -- (p5);
        \draw[densely dotted] (p2) -- (p3);
        \draw[densely dotted] (p2) -- (p4);
        \draw[densely dotted] (p2) -- (p5);

        \draw[->] (0,-1) -- (0,-2);
        \draw[->] (0,-1) -- (1,-2);
    \end{tikzpicture}
    \quad
    \begin{tikzpicture}
        \def\ngon{5}
        \node[regular polygon,regular polygon sides=\ngon,minimum size=2cm] (p) {};
        \foreach\x in {1,...,\ngon}{\node[draw, circle, fill=black, minimum size=0.2cm, inner sep=0pt] (p\x) at (p.corner \x){};}

        \draw[densely dotted] (p1) -- (p2);
        \draw[densely dotted] (p1) -- (p3);
        \draw[densely dotted] (p1) -- (p4);
        \draw[densely dotted] (p1) -- (p5);
        \draw (p2) -- (p3);
        \draw[densely dotted] (p2) -- (p4);
        \draw[densely dotted] (p2) -- (p5);

        \draw[->] (0,-1) -- (1,-2);
    \end{tikzpicture}
    \quad
    \begin{tikzpicture}
        \def\ngon{5}
        \node[regular polygon,regular polygon sides=\ngon,minimum size=2cm] (p) {};
        \foreach\x in {1,...,\ngon}{\node[draw, circle, fill=black, minimum size=0.2cm, inner sep=0pt] (p\x) at (p.corner \x){};}

        \draw[densely dotted] (p1) -- (p2);
        \draw[densely dotted] (p1) -- (p3);
        \draw[densely dotted] (p1) -- (p4);
        \draw[densely dotted] (p1) -- (p5);
        \draw (p2) -- (p3);
        \draw (p2) -- (p4);
        \draw (p2) -- (p5);
        
        \draw[->] (0,-1) -- (1,-2);
    \end{tikzpicture}
    \quad
    \begin{tikzpicture}
        \def\ngon{5}
        \node[regular polygon,regular polygon sides=\ngon,minimum size=2cm] (p) {};
        \foreach\x in {1,...,\ngon}{\node[draw, circle, fill=black, minimum size=0.2cm, inner sep=0pt] (p\x) at (p.corner \x){};}

        \draw (p1) -- (p2);
        \draw[densely dotted] (p1) -- (p3);
        \draw[densely dotted] (p1) -- (p4);
        \draw[densely dotted] (p1) -- (p5);
        \draw[densely dotted] (p2) -- (p3);
        \draw[densely dotted] (p2) -- (p4);
        \draw[densely dotted] (p2) -- (p5);

        \draw[->] (0,-1) -- (1,-2);
    \end{tikzpicture}
    \quad
    \begin{tikzpicture}
        \def\ngon{5}
        \node[regular polygon,regular polygon sides=\ngon,minimum size=2cm] (p) {};
        \foreach\x in {1,...,\ngon}{\node[draw, circle, fill=black, minimum size=0.2cm, inner sep=0pt] (p\x) at (p.corner \x){};}

        \draw (p1) -- (p2);
        \draw[densely dotted] (p1) -- (p3);
        \draw[densely dotted] (p1) -- (p4);
        \draw[densely dotted] (p1) -- (p5);
        \draw[densely dotted] (p2) -- (p3);
        \draw (p2) -- (p4);
        \draw (p2) -- (p5);

        \draw[->,white] (0,-1) -- (1,-2);
    \end{tikzpicture}
    \quad
    \begin{tikzpicture}
        \def\ngon{5}
        \node[regular polygon,regular polygon sides=\ngon,minimum size=2cm] (p) {};
        \foreach\x in {1,...,\ngon}{\node[draw, circle, fill=black, minimum size=0.2cm, inner sep=0pt] (p\x) at (p.corner \x){};}

        \draw (p1) -- (p2);
        \draw[densely dotted] (p1) -- (p3);
        \draw[densely dotted] (p1) -- (p4);
        \draw[densely dotted] (p1) -- (p5);
        \draw (p2) -- (p3);
        \draw (p2) -- (p4);
        \draw (p2) -- (p5);

        \draw[->] (0,-1) -- (0,-2);
    \end{tikzpicture}

    \vspace{0.3cm}

    \begin{tikzpicture}
        \def\ngon{5}
        \node[regular polygon,regular polygon sides=\ngon,minimum size=2cm] (p) {};
        \foreach\x in {1,...,\ngon}{\node[draw, circle, fill=black, minimum size=0.2cm, inner sep=0pt] (p\x) at (p.corner \x){};}

        \draw[densely dotted] (p1) -- (p2);
        \draw[densely dotted] (p1) -- (p3);
        \draw[densely dotted] (p1) -- (p4);
        \draw[densely dotted] (p1) -- (p5);
        \draw[densely dotted] (p2) -- (p3);
        \draw[densely dotted] (p2) -- (p4);
        \draw[densely dotted] (p2) -- (p5);
        \draw[densely dotted] (p3) -- (p4);
        \draw[densely dotted] (p3) -- (p5);
        \draw[densely dotted] (p4) -- (p5);
    \end{tikzpicture}
    \quad
    \begin{tikzpicture}
        \def\ngon{5}
        \node[regular polygon,regular polygon sides=\ngon,minimum size=2cm] (p) {};
        \foreach\x in {1,...,\ngon}{\node[draw, circle, fill=black, minimum size=0.2cm, inner sep=0pt] (p\x) at (p.corner \x){};}

        \draw[densely dotted] (p1) -- (p2);
        \draw[densely dotted] (p1) -- (p3);
        \draw[densely dotted] (p1) -- (p4);
        \draw[densely dotted] (p1) -- (p5);
        \draw[densely dotted] (p2) -- (p3);
        \draw[densely dotted] (p2) -- (p4);
        \draw[densely dotted] (p2) -- (p5);
        \draw (p3) -- (p4);
        \draw[densely dotted] (p3) -- (p5);
        \draw[densely dotted] (p4) -- (p5);
    \end{tikzpicture}
    \quad
    \begin{tikzpicture}
        \def\ngon{5}
        \node[regular polygon,regular polygon sides=\ngon,minimum size=2cm] (p) {};
        \foreach\x in {1,...,\ngon}{\node[draw, circle, fill=black, minimum size=0.2cm, inner sep=0pt] (p\x) at (p.corner \x){};}

        \draw[densely dotted] (p1) -- (p2);
        \draw[densely dotted] (p1) -- (p3);
        \draw[densely dotted] (p1) -- (p4);
        \draw[densely dotted] (p1) -- (p5);
        \draw (p2) -- (p3);
        \draw[densely dotted] (p2) -- (p4);
        \draw[densely dotted] (p2) -- (p5);
        \draw[densely dotted] (p3) -- (p4);
        \draw[densely dotted] (p3) -- (p5);
        \draw[densely dotted] (p4) -- (p5);
    \end{tikzpicture}
    \quad
    \begin{tikzpicture}
        \def\ngon{5}
        \node[regular polygon,regular polygon sides=\ngon,minimum size=2cm] (p) {};
        \foreach\x in {1,...,\ngon}{\node[draw, circle, fill=black, minimum size=0.2cm, inner sep=0pt] (p\x) at (p.corner \x){};}

        \draw[densely dotted] (p1) -- (p2);
        \draw[densely dotted] (p1) -- (p3);
        \draw[densely dotted] (p1) -- (p4);
        \draw[densely dotted] (p1) -- (p5);
        \draw (p2) -- (p3);
        \draw (p2) -- (p4);
        \draw (p2) -- (p5);
        \draw (p3) -- (p4);
        \draw (p3) -- (p5);
        \draw (p4) -- (p5);
    \end{tikzpicture}
    \quad
    \begin{tikzpicture}
        \def\ngon{5}
        \node[regular polygon,regular polygon sides=\ngon,minimum size=2cm] (p) {};
        \foreach\x in {1,...,\ngon}{\node[draw, circle, fill=black, minimum size=0.2cm, inner sep=0pt] (p\x) at (p.corner \x){};}

        \draw (p1) -- (p2);
        \draw[densely dotted] (p1) -- (p3);
        \draw[densely dotted] (p1) -- (p4);
        \draw[densely dotted] (p1) -- (p5);
        \draw[densely dotted] (p2) -- (p3);
        \draw[densely dotted] (p2) -- (p4);
        \draw[densely dotted] (p2) -- (p5);
        \draw[densely dotted] (p3) -- (p4);
        \draw[densely dotted] (p3) -- (p5);
        \draw[densely dotted] (p4) -- (p5);
    \end{tikzpicture}
    \quad
        \begin{tikzpicture}
        \def\ngon{5}
        \node[regular polygon,regular polygon sides=\ngon,minimum size=2cm] (p) {};
        \foreach\x in {1,...,\ngon}{\node[draw, circle, fill=black, minimum size=0.2cm, inner sep=0pt] (p\x) at (p.corner \x){};}

        \draw (p1) -- (p2);
        \draw[densely dotted] (p1) -- (p3);
        \draw[densely dotted] (p1) -- (p4);
        \draw[densely dotted] (p1) -- (p5);
        \draw (p2) -- (p3);
        \draw (p2) -- (p4);
        \draw (p2) -- (p5);
        \draw[densely dotted] (p3) -- (p4);
        \draw[densely dotted] (p3) -- (p5);
        \draw[densely dotted] (p4) -- (p5);
    \end{tikzpicture}
    \caption{All graphs options for edges labelled 1 or 2, showing each step as edges are added.}
    \label{1/2 options tree}
\end{figure}

    By swapping the assignment of 1 and 2 to edge styles, these graphs provide the remaining six graphs of \Cref{Graph options} up to permutation of the vertices. Therefore, we have shown that the nine graphs are the only possible options.

    \noindent\textbf{Constructing an array from a graph}

    As defined earlier, the edge values correspond to the number of missing entries in a cell of $A$, as well as the number of missing copies of a given symbol in a row or column. We will label the vertices of our graph as shown in \Cref{Graph numbering}, with 1 at the top and moving anti-clockwise.

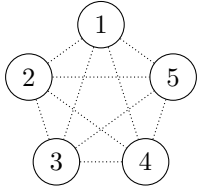
\begin{figure}[h]
    \centering
    \begin{tikzpicture}
        \def\ngon{5}
        \node[regular polygon,regular polygon sides=\ngon,minimum size=2cm] (p) {};
        \foreach\x in {1,...,\ngon}{\node[draw, circle] (p\x) at (p.corner \x){\x};}

        \draw[densely dotted] (p1) -- (p2);
        \draw[densely dotted] (p1) -- (p3);
        \draw[densely dotted] (p1) -- (p4);
        \draw[densely dotted] (p1) -- (p5);
        \draw[densely dotted] (p2) -- (p3);
        \draw[densely dotted] (p2) -- (p4);
        \draw[densely dotted] (p2) -- (p5);
        \draw[densely dotted] (p3) -- (p4);
        \draw[densely dotted] (p3) -- (p5);
        \draw[densely dotted] (p4) -- (p5);
    \end{tikzpicture}
    \caption{The vertex labelling used to construct arrays from the possible graphs.}
    \label{Graph numbering}
\end{figure}

    Note that each graph in \Cref{Graph options} has a different number of repetitions of 0, 1 and 2. Thus, we will refer to each graph by these numbers. Given one of these graphs, the aim is to construct an array $B$ which has the required number of symbols and repetitions in each cell, row and column. The first graph in \Cref{Graph options} has all 0 edges, which would be an empty array, thus we do not need an array in this case. For five of the other eight graphs, the corresponding array $B$ is shown in \Cref{rational solution for ALL 5 array A}. The arrays for the remaining three graphs can be constructed using these same arrays: take the union of $0^41^6$ and $1^{10}$ to get $1^42^6$, $0^11^9$ and $1^{10}$ to get $1^12^9$, and two copies of $1^{10}$ for $2^{10}$. The union of the arrays $A$ and $B$ provides a complete outline rectangle, and due to the entries in the cells $(i,i)$, this outline rectangle will lift to a $\LS(h_1\dots h_5)$ by \Cref{symROR to 2RP}.

    If the graph for a given partition is a vertex permutation of the ones given in \Cref{Graph options}, then the array given in \Cref{rational solution for ALL 5 array A} should have the rows, columns and symbols permuted the same way.

    \begin{figure}[h]
        \centering
        \begin{subfigure}[b]{0.23\textwidth}
        \centering
        \begin{tabular}{|l|l|l|l|l|}\hline
        \phantom{3} &  &  &  &  \\ \hline
         &  & 5 & 3 & 4 \\ \hline
         & 4 &  & 5 & 2 \\ \hline
         & 5 & 2 &  & 3 \\ \hline
         & 3 & 4 & 2 &  \\ \hline
        \end{tabular}
        \caption{$0^41^6$}
        \end{subfigure}
        \begin{subfigure}[b]{0.23\textwidth}
        \centering
        \begin{tabular}{|l|l|l|l|l|}\hline
        \phantom{3} &  & 5 & 3 & 4 \\ \hline
         &  & 4 & 5 & 3 \\ \hline
        4 & 5 &  & 2 & 1 \\ \hline
        5 & 3 & 1 &  & 2 \\ \hline
        3 & 4 & 2 & 1 &  \\ \hline
        \end{tabular}
        \caption{$0^11^9$}
        \end{subfigure}
        \begin{subfigure}[b]{0.23\textwidth}
        \centering
        \begin{tabular}{|l|l|l|l|l|}\hline
        \phantom{3} & 3 & 5 & 2 & 4 \\ \hline
        4 &  & 1 & 5 & 3 \\ \hline
        5 & 4 &  & 1 & 2 \\ \hline
        3 & 5 & 2 &  & 1 \\ \hline
        2 & 1 & 4 & 3 &  \\ \hline
        \end{tabular}
        \caption{$1^{10}$}
        \end{subfigure}

        \begin{subfigure}[b]{0.23\textwidth}
        \centering
        \begin{tabular}{|c|c|c|c|c|}\hline
        \phantom{3} & 4,5 & 2 & 2 & 3 \\ \hline
        3,5 &  & 1 & 1 & 4 \\ \hline
        4 & 1 &  & 5 & 2 \\ \hline
        2 & 3 & 5 &  & 1 \\ \hline
        2 & 1 & 4 & 3 &  \\ \hline
        \end{tabular}
        \caption{$2^11^9$}
        \end{subfigure}
        \quad
        \begin{subfigure}[b]{0.23\textwidth}
        \centering
        \begin{tabular}{|c|c|c|c|c|}\hline
        \phantom{3} & 4,5 & 2,4 & 3,5 & 2,3 \\ \hline
        3,5 &  & 1 & 1 & 4 \\ \hline
        4,5 & 1 &  & 2 & 1 \\ \hline
        2,3 & 1 & 5 &  & 1 \\ \hline
        2,4 & 3 & 1 & 1 &  \\ \hline
        \end{tabular}
        \caption{$2^41^6$}
        \end{subfigure}
    \caption{The arrays used for $B$ given a set of values for $y(i,j)$.}
    \label{rational solution for ALL 5 array A}
    \end{figure}
\end{proof}

We can use this result to settle the open cases in the following theorem, which can be found in Colbourn \cite{colbourn2018latin}.

\begin{theorem}
\label{Colb. incomplete}
    If $k\geq 5$, $\mu\geq 1$, and $3\mu\geq h_1\geq h_2\geq\dots\geq h_k\geq\mu$, then a $\LS(h_1\dots h_k)$ exists, except possibly when $k=5$ and $\mu = 6$.
\end{theorem}

Colbourn \cite{colbourn2018latin} also showed that partitions meeting this condition satisfy the necessary conditions given earlier.

\begin{lemma}
\label{Colbourn meets cond}
    If $k\geq 5$ and $(k-2)h_k\geq h_1\geq\dots\geq h_k>0$, then the conditions of \Cref{squarecondition1} and \Cref{squarecondition2} are satisfied.
\end{lemma}

We now complete the open cases of \Cref{Colb. incomplete} and obtain the following theorem.

\begin{theorem}
    If $k\geq 5$, $h_k>0$ and $h_1\leq 3h_k$, then a $\LS(h_1\dots h_k)$ exists.
\end{theorem}
\begin{proof}
    We need only consider the case when $k=5$. By \Cref{Colbourn meets cond}, we know that the conditions of \Cref{squarecondition2} are met, and thus the conditions of \Cref{All 5 parts} are met also.
\end{proof}

\section{Increasing a partition by a constant}

In \cite{kuhl2019existence}, Kuhl et al. provide the following theorem, however the proof does not cover all possible cases.

\begin{theorem}
    Let $q$ and $r_i$ be positive integers and $n_i=5q+r_i$, for each $i\in[5]$. If a $\LS(r_1\dots r_5)$ exists, then a $\LS(n_1\dots n_5)$ exists.
\end{theorem}

Within the proof, it is stated that a $\LS(q^4(q+r_i)^1)$ exists for all $q$, however this is only true when $q\geq \frac{r_i}{2}$ as given in \Cref{squaresatmost2}. Thus, the proof only holds when $q\geq \frac{r_i}{2}$ for all $i\in[5]$.

Although we have completely determined existence for partitions with five parts in the previous section, the proof of this result provides an interesting method of construction and we now prove a stronger result using a similar approach.

\begin{theorem}
\label{adding 1}
    If a $\LS(r_1r_2r_3r_4r_5)$ exists, then a $\LS(n_1n_2n_3n_4n_5)$ exists where $n_i = r_i + q$ for any positive integer $q$.
\end{theorem}
\begin{proof}
    We will first show that this statement holds when $q=1$, and then the result follows by induction.

    Since a $\LS(r_1r_2r_3r_4r_5)$ exists, there is an outline square for this realization. Let this outline square be $O$.
    Let $B$ be the array in Figure \ref{B}, and let $A$ be the latin square in Figure \ref{A}.
    \begin{figure}[h]
        \centering
        $$\arraycolsep=4pt\begin{array}{|c|c|c|c|c|} \hhline{|*{5}{-|}}
    1 & 4 & 2 & 5 & 3\\ \hhline{|*{5}{-|}}
    4 & 2 & 5 & 3 & 1\\ \hhline{|*{5}{-|}}
    2 & 5 & 3 & 1 & 4\\ \hhline{|*{5}{-|}}
    5 & 3 & 1 & 4 & 2\\ \hhline{|*{5}{-|}}
    3 & 1 & 4 & 2 & 5\\ \hhline{|*{5}{-|}}\end{array}$$
        \caption{The array $A$ as defined in \Cref{adding 1}.}
        \label{A}
    \end{figure}

    \begin{figure}[h]
        \centering
        \renewcommand{\arraystretch}{2.7}
        \begin{tabular}{|l|l|l|l|l|}\hline
        $1:2r_1$ & \thead{$3:r_2+r_3$\\$4:r_1-r_3$} & \thead{$4:r_3+r_4$\\$5:r_1-r_4$} & \thead{$2:r_1-r_5$\\$5:r_4+r_5$} & \thead{$2:r_2+r_5$\\$3:r_1-r_2$} \\\hline
        \thead{$3:r_1-r_2$\\$4:r_2-r_5$\\$5:r_2+r_5$} & $2:2r_2$ & $1:r_2+r_3$ & \thead{$1:r_1-r_3$\\$3:r_2+r_3+r_4-r_1$} & \thead{$3:r_2-r_4$\\$4:r_4+r_5$} \\\hline
        \thead{$2:r_1+r_3-r_4-r_5$\\$4:r_4+r_5$} & \thead{$1:r_2-r_3$\\$4:r_3-r_5$\\$5:r_3+r_5$} & $3:2r_3$ & \thead{$1:r_1+r_3-r_2-r_5$\\$2:r_2+r_4+r_5-r_1$} & $1:r_3+r_5$ \\\hline
        \thead{$2:r_2+r_4+r_5-r_1$\\$3:r_1-r_5$\\$5:r_1-r_2$} & \thead{$1:r_3+r_4$\\$5:r_2-r_3$} & \thead{$2:r_1-r_5$\\$5:r_3+r_4+r_5-r_1$} & $4:2r_4$ & \thead{$1:r_1-r_3$\\$3:r_3+r_4+r_5-r_1$}\\\hline
        \thead{$2:r_1-r_3$\\$3:r_2+r_3+r_5-r_1$\\$4:r_1-r_2$} & \thead{$1:r_1-r_4$\\$4:r_2+r_4+r_5-r_1$} & \thead{$1:r_1-r_2$\\$2:r_2+r_3+r_5-r_1$} & \thead{$1:r_2+r_4+r_5-r_1$\\$3:r_1-r_2$} & $5:2r_5$\\\hline
        \end{tabular}
        \caption{The array $B$ as used in \Cref{adding 1}.}
        \label{B}
    \end{figure}

    By taking $L = O\cup A\cup B$, we get an outline square for a $\LS(n_1n_2n_3n_4n_5)$.

    Cell $(i,j)$ of $L$ must contain $n_in_j = (r_i+1)(r_j+1) = r_ir_j + r_i + r_j + 1$ entries, which is $r_i+r_j+1$ more than the same cell of $O$. Similarly, row $i$ and column $i$ must each contain $n_in_j$ copies of symbol $j$. Thus, since $A$ has 1 copy of every symbol in every row and column, and 1 entry per cell, the array $B$ must have $r_i + r_j$ symbols in cell $(i,j)$ and $r_i + r_j$ copies of symbol $j$ in row or column $i$.

    Also note that the cells $(i,i)$ of $L$ will contain only symbol $i$ as required.

    The array $B$ satisfies these requirements, and all values are non-negative integers since $r_1\leq r_3+r_4+r_5$ by the existence of a $\LS(r_1r_2r_3r_4r_5)$.
\end{proof}

Given the simplicity of this result and the construction method used in the proof, it would be useful to have a similar result for partitions with $k\geq 6$. While we believe that the statement itself holds for such partitions, the method of construction cannot be extended the same way. We conclude by deriving a necessary condition for the existence of the array $B$ as used in \Cref{adding 1}, but for general $k$. We then show that this condition is only met in general when $k=6,7$.

\begin{lemma}
    For a partition $P = (h_1\dots h_k)$ of $n$, let $B$ be a $k\times k$ array of multisets such that
    \begin{itemize}
        \item cell $(i,j)$ contains $h_i + h_j$ symbols;
        \item row $i$ has $h_i + h_j$ copies of symbol $j$;
        \item column $i$ has $h_i + h_j$ copies of symbol $j$;
        \item and the entries in cell $(i,i)$ are only symbol $i$.
    \end{itemize}
    Then $(2k-2-3|D|)\sum_{j\in\overline{D}}h_j\geq (k+2-3|D|)\sum_{i\in D}h_i$ for all subsets $D\subseteq[k]$.
\end{lemma}
\begin{proof}
    Given a subset $D$ of $[k]$, we will consider the number of copies of symbol $a\in D$ in the cells $(i,j)$ for $i,j\in\overline{D}$ and compare this to the total number of symbols in those cells.

    Let $d=|D|$ and permute the parts of $P$ so that $D = [d]$. Consider the three subarrays shown in \Cref{subarrays}, where $E$ is a $d\times d$ array of multisets.

    \begin{figure}[h]
        \centering

        \begin{tikzpicture}
            \draw (0,0) -- (1.5,0);
            \draw (0,0) -- (0,-1.5);
            \draw (1.5,0) -- (1.5,-1.5);
            \draw (0,-1.5) -- (1.5,-1.5);
            \draw (0,-0.75) -- (1.5,-0.75);
            \draw (0.75,0) -- (0.75,-1.5);
            \node at (0.375,-0.375) {$E$};
            \node at (0.375,-1.125) {$F$};
            \node at (1.125,-1.125) {$G$};
        \end{tikzpicture}

        \caption{Subarrays of the array $B$.}
        \label{subarrays}
    \end{figure}
    
    In column $a$ of $E\cup F$ for $a\in D$, there are $h_a + h_b$ copies of symbol $b\in D$. Thus, there are $|D|h_a + \sum_{b\in D}h_b$ copies of symbols from $D$ in a column, and so there are $2|D|\sum_{a\in D}h_a$ copies of symbols from $D$ across all columns of $E\cup F$. Within the array $E$, the cells $(a,a)$ have exactly $2h_a$ copies of symbol $a\in D$. Therefore, in the subarray $F$ there are at most $2(|D|-1)\sum_{a\in D}h_a$ copies of symbols from $D$.

    Similarly, row $i$ of $F\cup G$ for $i\in\overline{D}$ has $|D|h_i + \sum_{b\in D}h_b$ symbols from $D$, and so there are $|D|\sum_{i\in\overline{D}}h_i + (k-|D|)\sum_{a\in D}h_a$ copies of symbols in $D$ across all of $F\cup G$. Thus, in the subarray $G$ there are at least $|D|\sum_{i\in\overline{D}}h_i + (k-3|D|+2)\sum_{a\in D}h_a$ copies of symbols from $D$. Noting that the cells $(i,i)$ for $i\in\overline{D}$ must only contain symbol $i$, there are $2(k-|D|-1)\sum_{j\in\overline{D}}h_j$ cells in $G$ which may contain symbols of $D$, and thus
    $$(2k-3|D|-2)\sum_{j\in\overline{D}}h_j \geq (k-3|D|+2)\sum_{a\in D}h_a.$$
\end{proof}

For a subset with $|D|=3$, this yields the necessary condition $(2k-11)\sum_{j\in\overline{D}}h_j \geq (k-7)\sum_{i\in D}h_i$. Take the partition $P = (h_1^31^{k-3})$ and $D = [3]$, with $h_1 = (2k-11)(k-3)$. A realization exists for the partition $P$ and for $((h_1+1)^32^{k-3})$ by \Cref{squaresatmost2} when $k\geq 5$. However, $(2k-11)(k-3) \leq (k-7)\cdot3(2k-11)(k-3)$ when $k\geq 8$. Therefore, it is clear that this array $B$ cannot always be constructed when $k\geq 8$, even when both a $\LS(h_1\dots h_k)$ and $\LS(n_1\dots n_k)$ exist for $n_i = h_i+1$.

When $k\in\{6,7\}$, it is possible that the array $B$ can be constructed for any partition which has a realization.

\begin{lemma}
    For $k\in\{6,7\}$, if a $\LS(h_1\dots h_k)$ exists then $(2k-2-3|D|)\sum_{j\in\overline{D}}h_j\geq (k+2-3|D|)\sum_{i\in D}h_i$ for all subsets $D\subseteq[k]$.
\end{lemma}
\begin{proof}
    Observe that for a subset $D'\subseteq[k]$, if $|D'| = k-|D|$, then $2k-3|D'|-2 = -(k-3|D|+2)$ and $k-3|D'|+2 = -(2k-3|D|-2)$. Thus, we need only consider subsets with $|D|\leq \frac{k}{2}$.
    
    Considering different values of $|D|$, the possible conditions to be satisfied for $k=6$ are:
    \begin{enumerate}
        \item $|D|=1$: $7\sum_{j\in\overline{D}}h_j \geq 5\sum_{i\in D}h_i$
        \item $|D|=2$: $2\sum_{j\in\overline{D}}h_j \geq \sum_{i\in D}h_i$
        \item $|D|=3$: $\sum_{j\in\overline{D}}h_j \geq -1\sum_{i\in D}h_i$
    \end{enumerate}
    Clearly (3) is always satisfied, since the sums are always positive. Since a realization exists, it is given by \Cref{squarecondition1} that $h_1\leq \sum_{i=3}^k h_i$, and so (1) is always true. Similarly, $h_1+h_2\leq 2\sum_{i=2}^k h_i$, which satisfies (2).

    The conditions for $k=7$ are:
    \begin{enumerate}
        \item $|D|=1$: $3\sum_{j\in\overline{D}}h_j \geq 2\sum_{i\in D}h_i$
        \item $|D|=2$: $2\sum_{j\in\overline{D}}h_j \geq \sum_{i\in D}h_i$
        \item $|D|=3$: $3\sum_{j\in\overline{D}}h_j \geq 0$
    \end{enumerate}
    All three conditions are satisfied by the same reasons used for $k=6$.
        
\end{proof}

\section*{Acknowledgement}
Funding: The author would like to acknowledge the support of the Australian Government through a Research Training Program (RTP) Scholarship.

\printbibliography

\end{document}